\documentclass[a4paper,fleqn,10pt]{amsart}

\usepackage[cp1250]{inputenc}
\usepackage[T1]{fontenc}

\usepackage{amssymb}
\usepackage[sort]{natbib}
\usepackage{mathrsfs}
\usepackage{mathscinet}
\usepackage{fullpage}
\usepackage{kamil}
\usepackage{float}
\usepackage{tikz}
\usepackage[colorlinks,linkcolor=blue,citecolor=blue]{hyperref}

\begin{document}
\bibliographystyle{plainnat}
\setcitestyle{numbers}

\title{Logarithmic asymptotics for multidimensional extremes under non-linear scalings}

\author{K.M.\ Kosi\'nski}
\address{Korteweg-de Vries Institute for Mathematics,
University of Amsterdam, the Netherlands; E{\sc urandom},
Eindhoven University of Technology}
\email{K.M.Kosinski@uva.nl}
\thanks{KK was supported by NWO grant 613.000.701.}

\author{M.\ Mandjes}
\address{Korteweg-de Vries Institute for Mathematics,
University of Amsterdam, the Netherlands; E{\sc urandom},
Eindhoven University of Technology, the Netherlands; CWI,
Amsterdam, the Netherlands}
\email{M.R.H.Mandjes@uva.nl}



\date{\today}

\begin{abstract}
Let $\bs W=\{\bs W_n:n\in\nn\}$ be a sequence of random vectors in $\rr^d$, $d\ge 1$. This paper considers the logarithmic asymptotics of the extremes of $\bs W$, that is, for any vector $\bs q>\0$ in $\rr^d$, we find
\[
\log\prob{\exists{n\in\nn}:\bs W_n> u \bs q}, \as u.
\]
We follow the approach of the restricted large deviation principle introduced in \citet{Duffy03}. That is, we assume that, for every $\bs q\ge\0$, and some scalings $\{a_n\},\{v_n\}$, $\frac{1}{v_n}\log\prob{\bs W_n/a_n\ge u \bs q}$ has a, continuous in $\bs q$, limit $J_{\bs W}(\bs q)$. We allow the scalings $\{a_n\}$ and $\{v_n\}$ to be regularly varying with a positive index. This approach is general enough to incorporate sequences $\bs W$, such that the probability law of $\bs W_n/a_n$ satisfies the large deviation principle with continuous, not necessarily convex, rate functions.
The formula for these asymptotics agrees with the seminal papers on this topic \citep{Glynn94,Duffield95,Collamore96,Duffy03}.
\end{abstract}

\maketitle

\section{Introduction}
\label{sec:introd}
Let $W=\{W_n:n\in\nn\}$ be a sequence of random variables taking values in $\rr$. Define $Q=\sup_{n\ge 0} W_n$. The random variable $Q$ has been extensively studied: if $W$ is time-reversible, then $Q$ has the same distribution as the steady-state workload distribution in a queue with free process $W$ (see, e.g., \citep{Reich58}); $Q$ has also various relations with finance and insurance risk. It is in general hard to determine the distribution of $Q$. One could therefore settle for the less ambitious goal of identifying the corresponding tail asymptotics, that is, finding a function $f$, such that $\prob{Q>u}\sim f(u)$, as $u\toi$ (i.e., the ratio of the two tends to 1 as $u\toi$). This, however, requires to impose a quite restrictive structure on $W$, even in the Gaussian setting. Therefore, one usually resorts to determining the \textit{logarithmic} asymptotics of the (right) tail of the distribution of $Q$. It has been observed that, in great generality,
\begin{equation}
\label{eq:heuristics}
\log\prob{Q>u}=\log\prob{\exists{n\in\nn}:W_n>u}
\sim
\log\sup_{n\ge0}\prob{W_n>u},\as u.
\end{equation}
The heuristic behind this claim is the \textit{principle of the largest term}: rare events occur in the most likely way. That is to say, if $W$ is unlikely to ever reach level $u$, then conditional on $W$ in fact reaching $u$, with overwhelming probability this happens near to the most likely epoch.

Using \eqref{eq:heuristics}, the tail behavior of $Q$ can be derived from the large deviation behavior of $W$; this was originally proposed by \citet{Kesidis93} and later made rigorous by \citet{Glynn94}. More formally, let $\Lambda$ be the limiting \textit{cumulant generating function} (CGF) of $W$, that is $\Lambda(\theta)=\lim_{n\toi}n^{-1}\log\ee\exp(\theta W_n)$, when the limit exists, and let $\Lambda$ satisfy the assumptions of the G\"artner--Ellis theorem. Then, by its virtue, the sequence of probability measures $\{\mu_n:n\in\nn\}$, where $\mu_n$ is the law of $W_n/n$, satisfies the large deviation principle (LDP) with rate function $\Lambda^*$, the Fenchel--Legendre transform of $\Lambda$ (also known as the convex conjugate of $\Lambda$);  see, \citep{Dembo98} for background on large deviations theory. Under these assumptions, \citep{Glynn94} asserts that
\begin{equation}
\label{eq:Glynn}
\lim_{u\toi}\frac{1}{u}\log\prob{\exists{n\in\nn}:W_n>u}=-\sup\{\theta:\Lambda(\theta)\le 0\}.
\end{equation}
Owing to its generality, this result is useful in a broad range of applications.

The result in \citep{Glynn94} has been extended in a notable paper by \citet{Duffield95} (see also \cite{Lelarge08}). The authors consider the logarithmic asymptotics of the tail of $Q$, by imposing assumptions on random variables $W_n/a_n$, where
$\{a_n:n\in\nn\}$ is some (non-necessarily linear) scaling. It is assumed that the \textit{scaled} limiting CGF of $W$, defined as $\Lambda(\theta)=\lim_{n\toi}v_n^{-1}\log\ee\exp(\theta v_n W_n/a_n)$ for some sequence $\{v_n:n\in\nn\}$, exists as an extended real number. The considered class of admissible scaling sequences is quite broad. It includes the case when $a\in\rv(A)$, $v\in\rv(V)$ are two regularly varying sequences with indexes $A,V>0$; this class is broad enough for most of the applications. Considering non-linear scalings allows to incorporate, for instance, long/short range dependent sequences corresponding to, for example, fractional Brownian motion. Similarly to \citet{Glynn94}, it is assumed that $\Lambda$ meets the assumptions of the G\"artner--Ellis theorem. Consequently, the sequence of probability measures $\{\mu_n:n\in\nn\}$, where $\mu_n$ is the law of $W_n/a_n$, satisfies the LDP with \textit{speed} $v_n$ and rate function $\Lambda^*$ (i.e., the convex conjugate of $\Lambda$). The main result from \citep{Duffield95} states that, under some additional assumptions,
\begin{equation}
\label{eq:Duffield}
\lim_{u\toi}\frac{1}{h(u)}\log\prob{\exists{n\in\nn}:W_n>u}=-\inf_{c>0} c^{-V/A}\Lambda^*(c),
\end{equation}
where $h= v\circ a^{-1}$ and $a^{-1}$ denotes the right inverse of $a$. It can be verified that \eqref{eq:Glynn} follows from \eqref{eq:Duffield} in case of the linear scaling: $a_n=v_n=n$.

Both \citep{Glynn94,Duffield95} impose additional conditions on $W$ only to infer that the sequence of probability measures $\{\mu_n:n\in\nn\}$, whether $\mu_n$ is the law of $W_n/n$ or $W_n/a_n$, satisfies the LDP with some well-behaved rate function. By exploiting the G\"artner--Ellis theorem, the considered class of possible rate functions is limited to convex functions, whereas, in general, it could be any lower semi-continuous function. Alternatively  therefore, one could assume that the LDP holds without any knowledge of how it was inferred. \citet{Duffy03} suggested a variant of this approach that allows to consider scaled sequences $W$ with nonconvex rate functions. More formally, the authors proposed the \textit{restricted} large deviation principle (RLDP) instead of the classical LDP. That is, they required that the limit
\begin{equation}
\label{eq:rldp}
\lim_{n\toi}\frac{1}{v_n}\log\prob{W_n/a_n>c}=-J_W(c)
\end{equation}
exists for every $c\ge 0$ and the function $J_W$ is continuous on the interior of the set upon which it is finite. Thus, if the LDP holds with a continuous (where finite) rate function $I_W$ (and speed $v_n$), then the RLDP holds with $J_W(c)=\inf_{x\ge c}I_W(x)$. 
Nevertheless it is not required that $I_W$ has been inferred from the G\"artner--Ellis theorem, nor, in principle, that the LDP holds at all. Observe that $J_W$ does not give any information about the behavior of $W$ for negative values so that in general the LDP is not even a prerequisite for the RLDP to hold. The main result from \citep{Duffy03}, under some additional assumptions, reads
\begin{equation}
\label{eq:Duffy}
\lim_{u\toi}\frac{1}{h(u)}\log\prob{\exists{n\in\nn}:W_n>u}=-\inf_{c>0} c^{-V/A}J_W(c).
\end{equation}
It can be verified, that if $J_W(c)=\inf_{x\ge c}I_W(x)$ and $I_W$ is convex, then \eqref{eq:Duffy} reduces to \eqref{eq:Duffield} in the special case when $I_W=\Lambda^*$.

Note that \eqref{eq:rldp} is a statement about the limiting behavior of $\prob{W_n/a_n>c}$ with $n$ growing large. This condition does not extract any information about $W$ for specific values of $n$, in particular the initial values of $W$. The distribution of $Q$ however, and hence the asymptotics as well, does involve the whole sequence $W$. It is therefore possible that the asymptotics of a \textit{single} $W_n$ could dominate those of $Q$. For instance, one can greatly alter $Q$ by simply substituting $W_0$ with a properly chosen heavy-tailed random variable $\tilde W_0$. To exclude such a scenario, \citet{Duffy03} introduced an additional, novel assumption referred to as the \textit{uniform individual decay rate} hypothesis; see \autoref{sec:MA}. Roughly speaking it prevents the sequence $W$ from having an ``unusual'' behavior for a single $W_n$. In fact, \citep[Section 4]{Duffy03} points out that this issue was overlooked by \citep{Duffield95}. That is, in order for the results from \citep{Duffield95} to hold (i.e., Equation \eqref{eq:Duffield}), one actually needs to impose further conditions. In the light of the generality of the result by \citet{Duffy03}, their paper can be treated as the most up to date treatment of the subject of logarithmic asymptotics for the supremum of a stochastic sequence.

In this paper we generalize and extend the result from \citep{Duffy03} in multiple ways. Firstly, in \autoref{thm:LDP} we show that, under the assumptions of \citep{Duffy03}, the sequence of probability measures $\{\mu_u^W:u\in\rr_+\}$, where $\mu_u^W(A)=\prob{\exists{n\in\nn}: W_n\in u A}$, satisfies the LDP with speed $h(u)$ and rate function $\tilde I_W$, such that $\tilde I_W(x)=x^{V/A}$, for $x\ge 0$, and $\tilde I_W(x)=\infty$, for $x<0$. In particular, for any $A\in\mathcal B(\rr_+)$,
\begin{equation}
\label{eq:Kos2}
\lim_{u\toi}\frac{1}{h(u)}\log\prob{\exists{n\in\nn}: W_n\in u A}=-
\inf_{x\in A} x^{V/A}\inf_{c>0} c^{-V/A}J_W(c).
\end{equation}
One can see that \eqref{eq:Kos2} extends \eqref{eq:Duffy} by setting $A=(1,\infty)$. \autoref{thm:LDP} is presented in \autoref{sec:extensionLDP}.

Furthermore, in \autoref{sec:MLDP} we allow the sequence $W$ to take values in $\rr^d$ for any $d\ge1$, rather than just $\rr$. As it turns out, this multidimensional setting imposes substantial additional challenges, as compared to the single-dimensional setting. Regarding notation, to explicitly distinguish the multidimensional case from the one-dimensional counterpart we will make use of the usual boldface fonts. That is, we write $\bs x$ for the vector $\bs x=(x_1,\ldots,x_d)$, where the dimension $d$ should be clear from the context. All vector relations should be understood coordinate-wise; for instance we write $\bs v\ge \bs w$ to mean $v_i\ge w_i$, for all $i = 1,\ldots,d$. With this notation, we consider a sequence $\bs W=\{\bs W_n:n\in\nn\}$ of random vectors in $\rr^d$. $\bs W$ is assumed to satisfy  multidimensional analogues of the assumptions from \citep{Duffy03}; see \autoref{sec:MA}. In particular, it is assumed that the (multidimensional) RLDP holds, that is, the limit
\[
\lim_{n\toi}\frac{1}{v_n}\log\prob{\bs W_n/a_n>\bs q}=-J_{\bs W}(\bs q),
\]
exists for any $\bs q\ge\bs 0$. Our second contribution, \autoref{thm:Mcase}, states that, for any vector $\bs q>\bs 0$,
\begin{equation}
\label{eq:Kos1}
\lim_{n\toi}\frac{1}{h(u)}\log\prob{\exists n\in\nn\:: \bs W_n>u\bs q}=-\inf_{c>0} c^{-V/A}J_{\bs W}(c\bs q).
\end{equation}
Obviously \eqref{eq:Kos1} is a generalization of \eqref{eq:Duffy} in the multidimensional sense.

Another significant contribution is that we also discuss various relations and connections with the existing literature. In \autoref{sec:sCGF} we present the relation between the RLDP approach, as undertaken here and in \citet{Duffy03}, and the approach via the cumulant generating functions, as undertaken in \citep{Glynn94,Duffield95}. In \autoref{sec:Collamore}, we discuss various results by \citet{Collamore96}, who considered a sequence of random vectors $\{\bs Y_n:n\in\nn\}$ in $\rr^d$, such that the sequence of probability measures $\{\mu_n:n\in\nn\}$, where $\mu_n$ corresponds to the law of $\bs Y_n/n$, satisfies the LDP with a convex rate function. He proved various LDP-like statements for the sequence of probability measures $\{\mu_u^{\bs Y,N}:u\in\rr_+\}$, where $\mu_u^{\bs Y,N}(A)=\prob{\exists{n\ge N}: \bs Y_n\in u A}$. His results, not referred to in \citep{Duffy03}, coincide with \eqref{eq:Duffy}--\eqref{eq:Kos1} in the case of $N=1$, linear scaling and convex rate function. We provide a discussion of these results also in \autoref{sec:Collamore}. 

We conclude our paper in \autoref{sec:examples}, where we present an extension of \autoref{thm:Mcase} from sequences $\{\bs W_n:n\in\nn\}$ to stochastic processes $\{\bs W_t:t\in\rr_+\}$. Furthermore, we apply our results to two examples. In the first one we treat heavy tailed processes which exhibit nonconvex rate functions, an example that was not covered by results that were known so far. In the second example, we compare our results with \citep{Debicki10b}, which addresses a similar problem for the case of $\bs W$ being Gaussian.

\section{Preliminaries}
In this paper we use the following notation. For a function $f:\rr^d\to\rr$ we denote its domain by $\dom f=\{\bs x:f(\bs x)<\infty\}$. As already introduced in \autoref{sec:introd}, we shall work with the following two functions $a,v:\rr_+\to\rr_+$, which throughout the whole paper are assumed to be regularly varying functions at infinity with positive indexes $A$ and $V$, respectively; we write $a\in\RV(A)$, $v\in\RV(V)$. It is well known that, for any regularly varying function $f\in\RV(F)$ with a positive index $F$, it is possible to construct a strictly increasing and continuous function $f'$ such that
\[
\lim_{t\toi}\frac{f'(ct)}{f'(t)}=\lim_{t\toi}\frac{f(ct)}{f(t)}
=\lim_{t\toi}\frac{f'(ct)}{f(t)}=c^F.
\]
Therefore, without loss of generality, we assume that both $a$ and $v$ are continuous and strictly increasing. We shall also speak about regularly varying sequences $a_n$ and $v_n$ defined by $a_n=a(n)$ and $v_n=v(n)$. Let $a^{-1}$ denote the inverse of $a$. Define a new function $h:\rr_+\to\rr_+$ by $h=v\circ a^{-1}$. The function $h$ belongs to the class $\RV(V/A)$. For details on regular variation, see, e.g., \citep{Bingham87}.

For any subset $A$ of $\rr^d$ we denote its cone by $\cone(A)=\{\lambda\bs x:\bs x\in A,\lambda\ge 0\}$, its closure by $\bar A$ and interior by $A^\circ$. For any convex function $f$ on $\rr^d$ we define its Fenchel--Legendre transform $f^*$ as $f^*(\bs x)=\sup_{\bs\alpha\in\rr^d}\left(\scal{\bs\alpha}{\bs x} - f(\bs\alpha)\right)$.

\subsection{Large deviations theory}
\label{sec:LDP}
We follow the definitions and setup as used in \citet{Dembo98}. All probability measures in this paper are assumed to be Borel measures. The function $I:\rr^d\to[0,\infty]$ is called a \textit{rate function} if $I$ is lower semi-continuous and $I\not\equiv\infty$. We say that $I$ is a \textit{good} rate function if, in addition, the level sets $\mathscr L_a I=\{\bs x:I(\bs x)\le a\}$ are compact for each $a\ge 0$ (which in fact is equivalent to the level sets being bounded as a function $f$ is lower semi-continuous if and only if its level sets are closed). A sequence of probability measures $\{\mu_n:n\in\nn\}$ on $\rr^d$ is said to satisfy the \textit{large deviation principle} (LDP) with \textit{rate function} $I$ if, for all $\Gamma\in\mathcal B(\rr^d)$,
\begin{equation}
\label{eq:LDPdef}
-\inf_{x\in \Gamma^\circ} I(x)
\le
\liminf_{n\toi}\frac{1}{n}\log\mu_n(\Gamma)
\le 
\limsup_{n\toi}\frac{1}{n}\log\mu_n(\Gamma)
\le
-\inf_{x\in \bar \Gamma} I(x).
\end{equation} 

Consider the empirical mean $\bs Z_n=\frac{1}{n}\sum_{i=1}^n\bs X_i$ of a sequence of i.i.d.~random vectors $\{\bs X_n:n\in\nn\}$ in $\rr^d$. Define $\mu_n$ as the law of $\bs Z_n$ and let $\Lambda_X=\log\ee e^{\scal{\bs\alpha}{\bs X_1}}$ be the \textit{cumulant generating function} (CGF) associated to the law of $\bs X_1$. In this case the classical theorem of Cram\'er applies.
\begin{theorem}[Cram\'er's theorem]
Assume that $0\in\intt{\dom {\Lambda_X}}$, then $\{\mu_n:n\in\nn\}$ satisfies the LDP on $\rr^d$ with good rate function $\Lambda_X^*$.
\end{theorem}

Now consider a general sequence of random vectors $\{\bs Z_n:n\in\nn\}$ in $\rr^d$, and let $\mu_n$ denote again the law of $\bs Z_n$. The CGF associated with the law $\mu_n$ is defined as $\Lambda_n(\bs\alpha)=\log\ee e^{\scal{\bs \alpha}{\bs Z_n}}$. Let $\Lambda(\bs\alpha)=\limsup_{n\toi}\frac{1}{n}\Lambda_n(n\bs\alpha)$ be the limiting CGF. With this notation, the following well-known theorem holds.
\begin{theorem}[G\"artner--Ellis theorem] Assume that
$0\in\intt{\dom \Lambda}$ and $\Lambda$ is an essentially smooth, lower semi-continuous function. Then $\{\mu_n:n\in\nn\}$ satisfies the LDP on $\rr^d$ with good rate function $\Lambda^*$.
\end{theorem}
Note that if $\bs Z_n=\frac{1}{n}\sum_{i=1}^n\bs X_i$ as in the setting of Cram\'er's theorem, then $\Lambda=\Lambda_X$ and further regularity conditions are not required.

One can consider LDPs with the so-called \textit{speed} $\{s_n:n\in\nn\}$, when $1/n$ in \eqref{eq:LDPdef} is replaced by $1/s_n\to0$. Then, the G\"artner--Ellis theorem remains valid if $\Lambda(\bs\alpha)=\limsup_{n\toi}\frac{1}{s_n}\Lambda_n(s_n\bs\alpha)$, the scaled limiting CGF, satisfies the assumptions. All results of this subsection carry through to continuous parameter families $\{\mu_u:u\in\rr_+\}$.

\subsection{Main assumptions}
\label{sec:MA}
In this paper we consider sequences $\bs W=\{\bs W_n:n\in\nn\}$ of random vectors in $\rr^d$, $d\ge 1$, satisfying the following assumptions.

\vb

\noindent\textbf{Restricted LDP Hypothesis.}
\textit{There exists a function $J_{\bs W}:\rr^d_+\to[0,\infty]$ such that, for every $\bs q\ge\0$,
\begin{equation}
\label{eq:RLDP}
\lim_{n\toi}\frac{1}{v_n}\log\prob{\frac{\bs W_n}{a_n}>\bs q}
=
- J_{\bs W}(\bs q).
\end{equation}}
\begin{remark}
\label{rem:RLDP}
If the sequence of probability measures $\{\mu_n:n\in\nn\}$, where $\mu_n$ denotes the law of $\bs W_n/a_n$, satisfies the LDP with speed $v_n$ and rate function $I_{\bs W}$, which is continuous where it is finite, then it also satisfies the RLDP hypothesis with $J_{\bs W}(\bs q)=\inf_{\bs x\ge\bs q} I_{\bs W}(\bs x)$ (hence the name of the hypothesis).
If $I$ is not continuous then it is easy to construct an example in which the restricted LDP does not hold. The opposite implication is also not true in general, that is, the restricted LDP hypothesis does not imply the LDP: property \eqref{eq:RLDP} does not provide any information about the negative values of $\bs W_n/a_n$.
\end{remark}

\vb

\noindent\textbf{Stability and continuity hypothesis.}
\textit{$J_{\bs W}( \bs 0)>0$ and there exists $\bs y>\bs 0$ such that $J_{\bs W}(\bs y)<\infty$. Furthermore, $J_{\bs W}$ is assumed to be continuous on $\intt{\dom{J_{\bs W}}}$.}

\vb

\noindent In the queueing context, $J_{\bs W}(\bs 0)>0$ is the usual stability condition. If $J_{\bs W}(\bs x)=\infty$ for all $\bs x\in\rr^d_+\setminus\{\bs 0\}$, then $\pp(\exists n\in\nn\: :\bs W_n>u\bs q)$ will have superexponential decay.

The restricted LDP hypothesis refers to limiting behavior of $\log\pp(\bs W_n>u a_n\bs q)$ for large $n$, not values for specific $n$. The asymptotic of a single $\bs W_n$ could dominate those of $\pp(\exists n\in\nn\: :\bs W_n>u\bs q)$. The condition below excludes this possibility.

\vb

\noindent\textbf{Uniform individual decay rate hypothesis.}
\textit{For a fixed vector $\bs q>\0$, there exist constants $F=F(\bs q)>V/A$ and $K=K(\bs q)>0$ so that for all $n$ and all $c>K$,
\[
\frac{1}{v_n}\log\prob{\frac{\bs W_n}{a_n}>c \bs q}\le - c^F.
\]
}
\begin{remark} The restricted LDP hypothesis, the stability and continuity hypothesis and the uniform individual decay rate hypothesis were originally introduced in \citep{Duffy03} in the one-dimensional case. That is, if one sets $d=1$, then all the above hypotheses reduce to those of \citep{Duffy03}. The hypotheses presented above can therefore be seen as natural extensions of the hypotheses from \citep{Duffy03} to the multidimensional setting.
\end{remark}
\begin{theorem}[One-dimensional case, Theorem 2.2, \citet{Duffy03}]
\label{thm:Duffy}
If the sequence $W=\{W_n:n\in\nn\}$ of random variables satisfies all the hypotheses of this subsection, then
\begin{equation}
\label{eq:thm:Duffy}
\lim_{u\toi}\frac{1}{h(u)}\log\prob{\exists{n\in\nn}: W_n>u}=-\inf_{c>0} c^{-V/A}J_W(c).
\end{equation}
\end{theorem}

\section{Two extensions}
\label{sec:MR}
In this section we present two generalizations of \autoref{thm:Duffy}. Firstly, we consider families of measures $\mu^W_u(A)=\prob{\exists{n\in\nn}: W_n\in u A}$ for a general set $A$; one easily sees that \autoref{thm:Duffy} considers the case of $A=(1,\infty)$. Secondly, we consider the situation when the sequence $W$ takes values in $\rr^d$, $d\ge 1$, rather than just $\rr$.

\subsection{Extension to the LDP}
\label{sec:extensionLDP}
Define a lower semi-continuous function $\tilde I_W$ by 
\[
\tilde I_W(x)=
\left\{
\begin{array}{cc}
\infty & \text{for}\quad x<0,\\
k_{J_W} x^{V/A} & \text{for}\quad x\ge0,
\end{array}
\right.
\]
where $k_{J_W}=\inf_{c>0} c^{-V/A}J_W(c)$ is the constant appearing on the right-hand side of \eqref{eq:thm:Duffy}. In this subsection we assume that the sequence of random variables $W=\{W_n:n\in\nn\}$ is such that $W_0=0$. This assures that $Q=\sup_{n\ge 0} W_n$ is a non-negative random variable. Let $\mu^W_u$ be a probability measures on $\rr$ defined as
$
\mu^W_u(A)=\prob{\exists{n\in\nn}: W_n\in u A}.
$
The following theorem is the main result of this subsection.
\begin{theorem}
\label{thm:LDP}
Under the assumptions of \autoref{thm:Duffy}, the family $\{\mu^W_u:u\in\rr_+\}$ satisfies the LDP with speed function $h$ and good rate function $\tilde I_W$.
\end{theorem}
\begin{proof}
First notice that, for any $k>0$,
\begin{equation}
\label{eq:intervals}
\lim_{u\toi}\frac{1}{h(u)}\log\mu^W_u((k,\infty))=-\tilde I_W(k).
\end{equation}
Indeed, 
\[
\frac{1}{h(u)}\log\mu^W_u((k,\infty))=
\frac{h(ku)}{h(u)}\frac{1}{h(ku)}\log\mu^W_{ku}((1,\infty)),
\]
so that \eqref{eq:intervals} follows from \autoref{thm:Duffy} combined with the fact that $h(ku)/h(u)$ tends to $k^{V/A}$.

Let $\Gamma$ be any set in $\mathcal B(\rr)$.
Now if for some $y$,  $\inf_{x\in \bar\Gamma}\tilde I_W(x)=\tilde I_W(y)$, then, for any $\eta\in(0,y)$, the monotonicity of $\tilde I_W$ on $\rr_+$ implies that
\[
\mu^W_u(\Gamma)\le\mu^W_u([y,\infty))\le\mu^W_u((y-\eta,\infty)).
\]
Hence, by \eqref{eq:intervals},
\[
\limsup_{u\toi}\frac{1}{h(u)}\log\mu^W_u(\Gamma)\le -\tilde I_W(y-\eta).
\]
This combined with the continuity of $\tilde I_W$ on $\rr_+$ implies the upper bound:
\[
\limsup_{u\toi}\frac{1}{h(u)}\log\prob{\exists{n\in\nn}: W_n\in u \Gamma}\le -\inf_{x\in \bar \Gamma}\tilde I_W(x).
\]
If $\inf_{x\in \bar \Gamma} \tilde I_W(x)=0$ or $\inf_{x\in \bar \Gamma}\tilde I_W(x)=\infty$, then the above bound holds trivially.

Now if $\Gamma^\circ\cap\rr_+=\emptyset$, then the lower bound
\[
\liminf_{u\toi}\frac{1}{h(u)}\log\prob{\exists{n\in\nn}: W_n\in u \Gamma}\ge -\inf_{x\in \Gamma^\circ}\tilde I_W(x)
\]
holds immediately. Therefore, let $\Gamma^\circ\subset\rr_+$ and $s\in \Gamma^\circ,\eta>0$ be such that $(s-\eta,s+\eta]\subset \Gamma^\circ$. Obviously, $\mu^W_u(\Gamma)\ge \mu^W_u((s-\eta,s+\eta])=
\mu^W_u((s-\eta,\infty))-\mu^W_u((s+\eta,\infty))$. From \eqref{eq:intervals}, for sufficiently large $u$, $\mu^W_u((s+\eta,\infty))\le\mu^W_u((s-\eta,\infty))/2$. Hence,
\[
\liminf_{u\toi}\frac{1}{h(u)}\log\mu^W_u(\Gamma)\ge 
\liminf_{u\toi}\frac{1}{h(u)}\log\left(\frac{\mu^W_u((s-\eta,\infty))}{2}\right)
=-\tilde I_W(s-\eta)\ge -\tilde I_W(s).
\]
Now the lower bound follows after optimization over $s\in \Gamma^\circ$. 
\end{proof}
\begin{remark}
The assumption that $W_0=0$ is natural and fulfilled in many applications. If it does not hold, however, then the LDP from \autoref{thm:LDP} remains true on $\mathcal B(\rr_+)$. In both of these cases, the continuity of $\tilde I_W$ on $\rr_+$ implies that for any set $A\in\mathcal B(\rr_+)$,
\[
\lim_{u\toi}\frac{1}{h(u)}\log\prob{\exists{n\in\nn}: W_n\in u A}
=-\inf_{x\in A} \tilde I_W(x).
\]
Hence, \autoref{thm:LDP} generalizes \autoref{thm:Duffy}, which can be seen by taking $A=(1,\infty)$.
\end{remark}
%

\subsection{Extension to the multidimensional case}
\label{sec:MLDP}
In this subsection we generalize \autoref{thm:Duffy} to the multidimensional case. \autoref{thm:Duffy} itself serves as a building block in the proof of the following theorem.
\begin{theorem}[Multidimensional case]
\label{thm:Mcase}
If the sequence $\bs W=\{\bs W_n:n\in\nn\}$ of random vectors in $\rr^d$, $d\ge 1$, satisfies the hypotheses of \autoref{sec:MA}, then, for any $\bs q>\bs 0$,
\[
\lim_{u\toi}\frac{1}{h(u)}\log\prob{\exists{n\in\nn}: \bs W_n>u\bs q}=-\inf_{c>0} c^{-V/A}J_{\bs W}(c\bs q).
\]
\end{theorem}
\begin{proof}
Notice that
$\prob{\bs W_n>u\bs q}=\prob{Z_n>u}$,
where $Z=\{Z_n:n\in\nn\}$ is a sequence of random variables such that $Z_n=\min_{i=1,\ldots,d}(W_{n,i}/q_i)$. The sequence $Z$ satisfies the assumptions of \autoref{thm:Duffy} with a function $J_Z$ given by $J_Z(c)=J_{\bs W}(c\bs q)$. Indeed, 
by the restricted LDP hypothesis, it follows that for every $c\ge 0$,
\[
\lim_{n\toi}\frac{1}{v_n}\log\prob{\frac{Z_n}{a_n}>c}= 
\lim_{n\toi}\frac{1}{v_n}\log\prob{\frac{\bs W_n}{a_n}>c\bs q}=
- J_{\bs W}(c\bs q).
\]
The stability and continuity hypothesis for $J_Z$ easily follows from the stability and continuity hypothesis for $J_{\bs W}$. The same applies to the uniform individual decay rate hypothesis. Therefore, \autoref{thm:Duffy} yields
\begin{align*}
\lim_{u\toi}\frac{1}{h(u)}\log\prob{\exists{n\in\nn}: \bs W_n>u\bs q}
&=
\lim_{u\toi}\frac{1}{h(u)}\log\prob{\exists{n\in\nn}: Z_n>u}\\
&=
- \inf_{c>0} c^{-V/A} J_Z(c)\\
&=
-\inf_{c>0} c^{-V/A} J_{\bs W}(c\bs q).
\end{align*}
This completes the proof.
\end{proof}

\section{Connections with existing literature} 
\label{sec:literature}
We have already discussed the relations of our results to \citep{Duffy03}.
In this section we shall discuss our findings in the light of already existing results from \citep{Glynn94,Duffield95,Collamore96}.

\subsection{The cumulant generating function approach}
\label{sec:sCGF}
The analyses in \citep{Glynn94,Duffield95} are based on the CGF. Both of them consider the case of $d=1$, but only \citep{Duffield95} allows for non-linear scaling. In this subsection we present conditions under which the main assumptions of the present paper are fulfilled. Recall that,
\[
\Lambda_n(\bs\alpha)=\log\ee\exp\left(\scal{\bs\alpha}{\frac{\bs W_n}{a_n}}\right)
\]
is the CGF of the law of $\bs W_n/a_n$. Here it is assumed that $\Lambda_n$ exists as a finite real number for all $\bs\alpha\in\rr^d$ and all $n\in\nn$. The hypotheses of our paper have simple expressions in terms of the CGF. The conditions we specify here for the CGF case, based on their one-dimensional analogues in \citep{Duffy03}, are intended for easy applicability, rather than maximum generality. Under these assumptions, the large deviation rate function is convex. The CGF technique is not applicable to models which have nonconvex rate functions.

\vb

\noindent\textbf{LDP hypothesis, CGF case.}
\textit{For each $\bs\alpha\in\rr^d$, the scaled limiting cumulant generating function $\Lambda(\bs\alpha)=\lim_{n\toi}\frac{1}{v_n}\Lambda_n(v_n\bs\alpha)$ exists. Furthermore, $\Lambda$ is assumed to be continuously differentiable.}

\vb

\noindent Under the above assumption, by the G\"artner--Ellis theorem, the sequence of probability measures $\{\mu_n:n\in\nn\}$, where $\mu_n$ is the law of $\bs W_n/a_n$, satisfies the LDP with rate function $I_{\bs W}=\Lambda^*$ and speed $v_n$. This implies that $I_{\bs W}$ is convex and continuous on the set where it is finite and therefore the RLDP holds with $J_{\bs W}(\bs q)=\inf_{\bs x\ge\bs q}\Lambda^*(\bs x)$, which is also continuous on $\intt{\dom{J_{\bs W}}}$. Hence, in order to assure that the stability and continuity hypothesis holds, one can require the following.

\vb

\noindent\textbf{Stability hypothesis, CGF case.}
\textit{There exists $\bs\alpha^\star>\0$ such that $\Lambda(\bs\alpha^\star)<0$.}

\vb

\noindent The above hypothesis implies that $J_{\bs W}(\0)\ge - \Lambda(\bs\alpha^\star)>0$. Indeed,
\[
J_{\bs W}(\0)=\inf_{\bs x\ge \0}\sup_{\bs \alpha\in\rr^d}\left(\scal{\bs \alpha}{\bs x}-\Lambda(\bs\alpha)\right)
\ge 
\sup_{\bs\alpha\in\rr^d_+}\inf_{\bs x\ge \0}\left(\scal{\bs\alpha}{\bs x}-\Lambda(\bs\alpha)\right)
=
-\inf_{\bs\alpha\in\rr^d_+}\Lambda(\bs\alpha)\ge -\Lambda(\bs \alpha^\star).
\]

\vb

\noindent\textbf{Uniform individual decay rate hypothesis, CGF case.}
\textit{There exist constants $F'$ and $M$ such that $F'>\max\{V/A,1\}$ and
$
\frac{1}{v_n}\Lambda_n(v_n\bs\alpha)\le M \|\bs \alpha\|^{F'/(F'-1)}
$
for all $\bs\alpha>\0$ and all $n\in\nn$.}

\vb

\noindent Under this hypothesis, for each $F\in(1,F')$, there exists a constant $K_F=K_F(\bs q)$ such that, for all $c>K_F$, and all $n\in\nn$,
\[
\frac{1}{v_n}\log\prob{\frac{\bs W_n}{a_n}>c\bs q}\le -c^F.
\]
That is the uniform individual decay rate hypothesis certainly holds.
Indeed, an elementary consequence of Chernoff's inequality is
\[
\log\prob{\bs W_n>ca_n\bs q}\le -v_n\left(c\scal{\bs\alpha}{\bs q}-\frac{1}{v_n}\Lambda_n(v_n\bs\alpha)\right),
\]
for any $\bs\alpha>\0$. It then follows that
\[
\log\prob{\bs W_n>ca_n\bs q}\le -v_n\left(c\scal{\bs\alpha}{\bs q}-M \|\alpha\|^{F'/(F'-1)}\right).
\]
Choosing $\bs \alpha=(c(F'-1)\|\bs q\|/(MF'))^{F'-1}\bs q$, we have
\begin{equation}
\label{eq:dummy34}
\log\prob{\bs W_n>ca_n\bs q}\le -v_n (c\|q\|^2)^{F'} 
\left(
M^{1-F'}{F'}^{-F'}(F'-1)^{F'-1}
\right).
\end{equation}
Since $M$ and $F'$ are constants, for each $F\in(\max\{V/A,1\},F')$, there exists $K_F=K_F(\bs q)$ such that, for all $c>K_F$, the right hand side of \eqref{eq:dummy34} will be less than $-v_n c^F$.

\subsection{The LDP with a convex rate function}
\label{sec:Collamore}
The purpose of this subsection is to discuss the differences between the approach from \citep{Collamore96} and the one in our paper.

\citet{Collamore96} considered a sequence $\bs Y=\{\bs Y_n:n\in\nn\}$ of random vectors in $\rr^d$. The main assumption of his paper is that the sequence of probability measures $\{\mu_n:n\in\nn\}$, where $\mu_n$ is the law of $\bs Y_n/n$, satisfies the LDP with a convex rate function $I_{\bs Y}$, such that $\mathscr L_0 I_{\bs Y}\ne\emptyset$, that is $I_{\bs Y}(\bs x)=0$ for some $\bs x\in\rr^d$. Furthermore, it is assumed that (using the notation from \citep{Collamore96}):
\begin{itemize}
\item[\textbf{H1}':] $\sup_{n\ge N}\frac{1}{n}\log\ee\exp\scal{\bs\alpha} {\bs Y_n}<\infty$ for all $\bs\alpha\in\mathscr L_0 I^*_{\bs Y}$ and $N$ greater than or equal to some $N_0$.
\end{itemize}
\begin{remark}
\label{rem:condH1}
\textbf{H1}' is a regularity condition on the sequence $\bs Y$. It is satisfied if $\bs Y$ is, for example, the $n$th partial sum of an i.i.d.~sequence of random vectors satisfying the condition given in Cram\'er's theorem, or the weaker condition given in \citep{Ney95}. It also holds when $\bs Y$ is a Markov-additive process satisfying the uniform recurrence condition (6.2) in \citep{Ney87}. In both these cases $N_0$ can be taken to be $1$. $\bs Y$ can also be a general sequence satisfying the conditions of the G\"artner--Ellis theorem and (i) $\Lambda(\bs\alpha)$ is finite in the neighborhood of each $\bs\alpha\in\mathscr L_0 \Lambda$; (ii) the level sets of $\Lambda$ are compact; recall that
$\Lambda(\bs\alpha)=\limsup_{n\toi}\frac{1}{n}\log\ee\exp\scal{\bs\alpha} {\bs Y_n}$.
\end{remark}
\begin{theorem}[\citet{Collamore96}, Theorem 2.2]
\label{thm:collamore}
Suppose $A$ is a general set in $\rr^d$ and \textbf{H1}' and 
\begin{itemize}
\item[\textbf{H2}:] for some $\delta>0$, $A\cap\cone(\mathscr C_\delta)=\emptyset$, where
$
\mathscr C_\delta=\{\bs x:\inf_{\bs y\in\mathscr L_0 I_{\bs Y}}\|\bs x-\bs y\|<\delta\};
$
\end{itemize}
are satisfied. Then, for any $N\ge N_0$,
\[
\liminf_{u\toi}\frac{1}{u}\log\prob{\exists{n\ge N}:\bs Y_n\in u A}\ge -\inf_{\bs x\in A^\circ} \tilde I_{\bs Y}(\bs x)
\]
and
\[
\limsup_{u\toi}\frac{1}{u}\log\prob{\exists{n\ge N}:\bs Y_n\in u A}\le -\inf_{\bs x\in \bar A} \tilde I_{\bs Y}(\bs x),
\]
where $\tilde I_{\bs Y}(\bs x)=\sup_{\bs \alpha\in\mathscr L_0 I^*_{\bs Y}}\scal{\bs \alpha}{\bs x}$ is the support function of $\mathscr L_0 I^*_{\bs Y}$.
\end{theorem}

\begin{remark}
\label{rem:condH2}
Condition \textbf{H2} is an admissibility condition on sets $A$. Recall that $\mathscr L_0 I_{\bs Y}$ is the set of all the points $\bs y$ for which $I_{\bs Y}(\bs y)=0$. Intuitively, these are the points of the typical behavior of $\bs Y$. Recall also that in the setting of Cram\'er's theorem, when $\bs Y_n=\bs X_1+\ldots+\bs X_n$, for an i.i.d.~sequence of random vectors $\{\bs X_n:n\in\nn\}$ in $\rr^d$, $\mathscr L_0 I_{\bs Y}=\{\ee \bs X_1\}$. Therefore, the set $\mathscr C_\delta$ can be thought of as the $\delta$-neighborhood of all such points of typical behavior of $\bs Y$, and thus $A$ can be any general set that avoids the ``central tendency'' $\cone(\mathscr C_\delta)=\{\lambda \bs x:\lambda>0,\bs x\in\mathscr C_\delta\}$ of $\bs Y$.
\end{remark}

The straightforward major differences between the current approach and the one from \citep{Collamore96} are the following. \citep{Collamore96} considers the multidimensional case and sets satisfying \textbf{H2}, but only allows linear scaling. The sequence $\bs Y$ has to satisfy \textbf{H1}' and the sequence of measures corresponding to $\bs Y_n/n$, the LDP with \textit{convex} rate function. In our setup we considered the multidimensional case and regularly varying scalings, general sets in the case of $d=1$, but only quadrants $\{\bs x\in\rr^d:\bs x>\bs q\}$, for any $\bs q>\0$, in the case of $d>1$. Furthermore, we do not require the LDP to hold, but impose the restricted LDP hypothesis allowing for \textit{continuous} rate functions. We have already explained that non-linear scalings allow to incorporate, for instance, long/short range dependent sequences stemming from, for example, fractional Brownian motion. Also, as explained in \autoref{rem:RLDP}, it is possible that the restricted LDP holds when the LDP does not and vice versa. 

Observe that, if $I_{\bs Y}$ is continuous where finite, then the restricted LDP holds with $J_{\bs Y}(\bs q)=\inf_{\bs x\ge\bs q}  I_{\bs Y}(\bs x)$. By \citep[Theorem 13.5]{Rockafellar70}, $\tilde I_{\bs Y}(\bs x)$ is equal to the closure of $L(\bs x)=\inf_{\tau>0}\tau^{-1} I_{\bs Y}(\tau\bs x)$, that is, the greatest lower semi-continuous function majorized by $L$. Furthermore, if $A=\{\bs x\in\rr^d:\bs x>\bs q\}$, for some $\bs q>\0$, satisfies \textbf{H2}, then the upper and the lower bound in \autoref{thm:collamore} are equal to
\[
\inf_{\bs x\ge\bs q} \tilde I_{\bs Y}(\bs x)=
\inf_{\bs x\ge\bs q} \inf_{\tau>0}\tau^{-1} I_{\bs Y}(\tau\bs x)=
\inf_{\tau>0}\tau^{-1}\inf_{\bs x\ge\bs q}  I_{\bs Y}(\tau\bs x)=
\inf_{\tau>0}\tau^{-1} J_{\bs Y}(\tau\bs q).
\]
Hence, if, in addition, we are allowed to take $N_0=1$, then \autoref{thm:collamore} coincides with \autoref{thm:Mcase} in the case of linear scaling and convex rate functions.

\section{Examples}
\label{sec:examples}
In this section we discuss some of the examples in which the theory of this paper can be applied. Firstly, we shall consider an example that does not fit in frameworks of any of the previous paper. This is due to its multidimensional nature and the fact that the rate function appearing there is nonconvex. Secondly, we shall consider a multidimensional Gaussian example and shall try to recover some previously known results.

Let us begin with discussing an extension of \autoref{thm:Mcase} from sequences $\{\bs W_n:n\in\nn\}$ to stochastic processes $\{\bs W_t:t\in\rr_+\}$. To this end we formulate an additional necessary hypothesis.

\vb

\noindent\textbf{Extension hypothesis.}
\[
\lim_{u\toi}\sup_{n\in\nn}\frac{\log\prob{\exists{t\in(n,n+1]}:\bs W_t>u\bs q}}{h(u)}
=
\lim_{u\toi}\sup_{n\in\nn}\frac{\log\prob{\bs W_n>u\bs q}}{h(u)}
\]

\vb

\noindent The above hypothesis was also introduced in \citep{Duffy03} in the one-dimensional case. Therein, it is argued that under this hypothesis, for the case of $d=1$, \autoref{thm:Duffy} extends from sequences $\{W_n:n\in\nn\}$ to processes $\{W_t:t\in\rr_+\}$. It is straightforward to conclude from the proof of \autoref{thm:Mcase} that this is also the case if $d>1$.

\subsection{Application to heavy tailed processes}
Let us consider  processes of the type described in \citet[Section 3.2]{Duffy03}. To this end, we first introduce the heavy tailed distribution $\mathcal H$
by
\[
\prob{\mathcal H\ge x}=l(x)e^{-v(x)},
\]
where $l$ is a slowly varying function and $v\in\RV(V)$ with $V\in(0,1)$. 

Now consider a continuous time process $\{Y_t:t\in\rr_+\}$ taking the values $0$ and $1$, with the times spent in the 0 and 1 states being a sequence of i.i.d.~random variables with the same distribution as $\mathcal W$. For $c>0$, define a process $\{Z_t:t\in\rr_+\}$ via
\begin{equation}
\label{eq:defW}
Z_t=\int_0^t(Y_s-c){\rm d}s.
\end{equation}
It was shown by \citet{Duffy08} that the family of probability measures $\{\mu_{t}:t\in\rr_+\}$, where $\mu_t$ is the law of $Z_t/t$, satisfies the LDP with speed function $v(t)=t^V$ and
the (nonconvex!) rate function
\begin{equation}
\label{eq:ratef}
I_c(x) = 
\left\{
\begin{array}{ll}
(1-2(x+c))^V&\mbox{if $x\in [-c,\frac{1}{2}-c],$}\\
(2(x+c)-1)^V&\mbox{if $x\in[\frac{1}{2}-c,1-c],$}\\
\infty&\mbox{otherwise.}
\end{array}
\right.
\end{equation}

Now consider two independent processes $\{Y^{1}_t:t\in\rr_+\}$ and $\{Y^{2}_t:t\in\rr_+\}$ defined as above and construct the corresponding processes $Z^{1}=\{Z^{1}_t:t\in\rr_+\}$ and $Z^2=\{Z^{2}_t:t\in\rr_+\}$ via \eqref{eq:defW} with constants $c_1>0$ and $c_2>0$, respectively. Now let $\bs Z=\{\bs Z_t:t\in\rr_+\}$, where $\bs Z_t=(Z^1_t,Z^2_t)$. Note that the rate function $I_{\bs Z}$ corresponding to $\bs Z$ is given by $I_{\bs Z}(\bs x)=I_{c_1}(x_1)+I_{c_2}(x_2)$, where $I_{c_1}$ and $I_{c_2}$ are given by \eqref{eq:ratef}. Finally, define a new process $\bs W=\{\bs W_t:t\in\rr_+\}$ via $(W^1_t,W^2_t)=(Z^1_t,Z^1_t+Z^2_t)$. According to the contraction principle \citep[Theorem 4.2.1]{Dembo98}, the family of probability measures $\{\mu^{\bs W}_t:t\in\rr_+\}$ on $\rr^2$, where $\mu_t^{\bs W}$ is the law of $\bs W_t/t$, satisfies the LDP with speed $v(t)=t^V$ and rate function $I_{\bs W}$ given by
\[
I_{\bs W}(\bs x)=\inf_{\bs v\in\rr^2:(v_1,v_2)=(x_1,x_2-x_1)}\left(I_{c_1}(v_1)+I_{c_2}(v_2)\right)=I_{c_1}(x_1)+I_{c_2}(x_2-x_1).
\] 
Hence $\bs W$ satisfies the restricted LDP hypothesis with rate function $J_{\bs W}(\bs q)=\inf_{\bs x\ge \bs q} I_{\bs W}(\bs x)$, $\bs q\ge\0$.

Let $\hat q_1=q_1+c_1$ and $\hat q_2=q_2+c_1+c_2$. Elementary calculus reveals that, with $\bs q\ge \0$, 
%
%
%
$$
J_{\bs W}(\bs q)=
\left\{
\begin{array}{l}
J_{\bs W}^1(\bs q)\\
J_{\bs W}^2(\bs q)\\
J_{\bs W}^3(\bs q)\\
\infty
\end{array}
\right.
=
\left\{
\begin{array}{l}
I_{\bs W}(q_1,q_1+\half-c_2)\\
I_{\bs W}(q_2+c_2-\half,q_2)\\
I_{\bs W}(1-c_1,q_2)\\
\infty
\end{array}
\right.
=
\left\{
\begin{array}{ll}
I(\hat q_1)&\mbox{if $\hat q_1\in(c_1, 1], \hat q_2\in(c_1+c_2, \half +\hat q_1]$,}\\
I(\hat q_2-\half )&\mbox{if $\hat q_1\in(c_1, 1], \hat q_2\in[ \half+\hat q_1,\frac{3}{2}]$,}\\
1+I(\hat q_2-1)&\mbox{if $\hat q_1\in(c_1, 1], \hat q_2\in[\frac{3}{2},2]$,}\\
\infty&\mbox{if $\hat q_1>1$ or $\hat q_2>2$,}
\end{array}
\right.
$$
where $I=I_0$. Recall that $c_1,c_1\in(\half,1)$, thus we do not define $J_{\bs W}(\bs q)$ for $\hat q_1\le\half$, $\hat q_2\le 1$. Note also that, if $c_1+c_2\in[\frac{3}{2},2)$, then $\hat q_2\ge\frac{3}{2}$, so that the first two cases in the definition of $J_{\bs W}$ do not occur. Finally, one can easily check that $J_{\bs W}$ is continuous on the interior of the set where it is finite and that $J_{\bs W}(\bs 0)>0$. Hence the stability and the continuity hypothesis holds.
\definecolor{L-gray}{gray}{0.95}
\definecolor{LL-gray}{gray}{0.9}
\definecolor{LLL-gray}{gray}{0.85}

{\footnotesize
\begin{figure}[H]
\bigskip
\begin{tikzpicture}
\draw[->] (0,-0.67)--(0,5);
\draw[->] (-0.5,0)--(5,0);
\draw (5,0) node[below] {$q_1$};
\draw (0,5) node[left] {$q_2$};

\draw (0,4) --(4,4);
\draw (0,4) node[left] {\scriptsize $2-c_1-c_2$};

\draw (0,1.33) -- (5,1.33);
\draw (0,1.33) node[left] {\scriptsize $\frac{3}{2}-c_1-c_2$};
\draw (4,0) node[below]{\scriptsize $1-c_1$};

\draw [black, fill=L-gray] (0,1.33) -- (4,1.33) -- (4, 4) -- (0,4) -- (0,1.33);

\draw [black, fill=LL-gray] (0,0) -- (2.67,0) -- (4, 1.33) -- (0,1.33) -- (0,0);

\draw [black, fill=LLL-gray]  (2.67,0) -- (4, 1.33) -- (4,0) -- (2.67,0);

\draw (4.09, 0.35) node[left]{\scriptsize $J^1_{\boldsymbol W}({\boldsymbol q})$};
\draw (2.39, 0.65) node[left]{\scriptsize $J^2_{\boldsymbol W}({\boldsymbol q})$};
\draw (2.39, 2.65) node[left]{\scriptsize $J^3_{\boldsymbol W}({\boldsymbol q})$};

\draw (4,0) -- (4,4);

\draw(2,-0.67) -- (5,2.33);

\draw (2.5,-1.25) node[below] {Fig.\ 1: Definition of $J_{\boldsymbol W}({\boldsymbol q})$; case $c_1+c_2<\frac{3}{2}$. The rate function is $\infty$ for $q_1,q_2>0$ outside the shaded regions.\:\:\:};
\end{tikzpicture}
\end{figure}}

The coordinates of $\bs W$ are bounded above (by $1-c_1$ and $2-c_1-c_2$ respectively), the uniform decay rate hypothesis is satisfied. Note also that, for every $n$ and $t\in(n,n+1]$, $\bs W_t=\bs W_n+\bs R_{n,t}$, where $\bs R_{n,t}=(R_{n,t}^1, R_{n,t}^1)$ with both $R^1_{n,t}=\int_n^t(Y_s^1-c_1)\D s$ and $R^2_{n,t}=\int_n^t(Y_s^1+Y_s^2-c_1-c_2)\D s$ bounded independently of $t$. This immediately implies the extension hypothesis too. Therefore, the extended version of \autoref{thm:Mcase} applies. It gives, for every $\bs q>\0$,
$$
\lim_{u\toi}\frac{1}{u^V}\log\prob{\exists{t\in\rr_+}:\int_0^t(Y_s^1-c_1)\D s>uq_1,\int_0^t(Y_s^1+Y_s^2-c_1-c_2)\D s>uq_2}
=-\inf_{t\ge 0}\frac{J_{\bs W}(t\bs q)}{t^V}.
$$
Let $h(t)=\frac{J_{\bs W}(t\bs q)}{t^V}$, then with self-evident notation,
if  $c_1+c_2\in[\frac{3}{2},2)$,
$$
\inf_{t\ge 0}h(t)=
\left\{
\begin{array}{ll}
h^3\left(\frac{1-c_1}{q_1}\right)\\
h^3\left(\frac{2-c_1-c_2}{q_2}\right)
\end{array}
\right.
=
\left\{
\begin{array}{ll}
\left(\frac{q_1}{1-c_1}\right)^V
\left(1+\left[2\frac{q_2}{q_1}(1-c_1)+2(c_1+c_2)-3\right]^V\right),
&\frac{q_1}{q_2}\ge\frac{1-c_1}{2-c_1-c_2},
\\
2\left(\frac{q_2}{2-c_1-c_2}\right)^V,
&\frac{q_1}{q_2}\le\frac{1-c_1}{2-c_1-c_2},
\end{array}
\right.
$$
and if $c_1+c_2\in(1,\frac{3}{2})$,
\[
\inf_{t\ge 0}h(t)=
\left\{
\begin{array}{ll}
h^1\left(\frac{1-c_1}{q_1}\right)\\
h^2\left(\frac{\frac{3}{2}-c_1-c_2}{q_2}\right)
\end{array}
\right.
=
\left\{
\begin{array}{ll}
\left(\frac{q_1}{1-c_1}\right)^V
&\frac{q_1}{q_2}\ge\frac{1-c_1}{\frac{3}{2}-c_1-c_2},\\
\left(\frac{q_2}{\frac{3}{2}-c_1-c_2}\right)^V
&\frac{q_1}{q_2}\le\frac{1-c_1}{\frac{3}{2}-c_1-c_2}.\\
\end{array}
\right.
\]
\subsection{Application to Gaussian processes}
In this subsection we consider an example from \citet[Section 3.2]{Debicki10b}.
That is, let $\bs Y=\{\bs Y(t):t\in\rr_+\}$ be a centered Gaussian process in $\rr^d$ with stationary increments and covariance matrix $\Sigma_t=\diag(c_1\sigma^2(t),\ldots,c_d\sigma^2(t))$, so that the coordinates of $\bs Y(t)=( Y_1(t),\ldots, Y_d(t))$ are independent and, for each $i=1,\ldots,d$, $\Var(Y_i(t))=c_i\sigma^2(t)$, for some $c_i>0$, and $\sigma^2\in\RV(\gamma)$, where $\gamma\in(0,2)$; compare these assumptions with assumptions \textbf{C1}-\textbf{C3} of \citep[Section 3.2]{Debicki10b}. 

For an invertible matrix $S$, define a new Gaussian process $\bs W=\{\bs W(t): t\in\rr_+\}$ in $\rr^d$ via $\bs W(t)=S\bs Y(t) - \bs i(t)$, where $\bs i:\rr\to\rr^d$ is such that $\bs i(t)=(t,\ldots,t)$. Set $a(t)=t$ and $v(t)=t^2\sigma^{-2}(t)$ and notice that, 
\begin{align*}
\Lambda_t(\bs\alpha)&=
\log\ee\exp\left(\scal{\bs\alpha}{\frac{\bs W(t)}{a(t)}}\right)
=
\frac{1}{2}\scal{\bs\alpha}{\frac{S\Sigma_t S^T}{t^2}\bs\alpha}-\scal{\bs\alpha}{\bs i(1)},
\end{align*}
so that the CGF variant of the uniform individual decay rate hypothesis holds with $F'=2$.
Furthermore,
\[
\Lambda(\bs\alpha)=\lim_{t\toi}\frac{1}{v(t)}\Lambda_t(v(t)\bs\alpha)=
\frac{1}{2}\scal{\bs\alpha}{S C S^T\bs\alpha}-\scal{\bs\alpha}{\bs i(1)},
\]
where $C=\diag(c_1,\ldots,c_d)$, so that the CGF variant of the LDP hypothesis holds and the RLDP is satisfied with
\[
J_{\bs W}(\bs q)=\inf_{\bs x\ge \bs q}\Lambda^*(\bs x)=
\frac{1}{2}\inf_{\bs x\ge \bs q}\scal{S^{-1}(\bs x+\bs i(1))}{ C^{-1} S^{-1}(\bs x+\bs i(1))}, 
\]
where the form of $\Lambda^*$ follows from \citet[Theorem 12.3]{Rockafellar70}. 
From this formula it follows that the stability and continuity hypothesis holds. 
Finally, by the stationarity of increments of $\bs Y$, for any $\varepsilon>0$,
\begin{equation}
\label{eq:1}
\prob{\exists{t\in(n,n+1]}: \bs W(t)>u\bs q}
\le
\prob{\exists{t\in[0,1]}: \bs W(t)>u\varepsilon\bs q}+\prob{\bs W(n)>u(1-\varepsilon)\bs q}.
\end{equation}
Now, for any $\bs x>\0$, define a new Gaussian process via $Z(t)=\frac{\scal{\bs SY(t)}{\bs x}}{\scal{\bs x}{\varepsilon\bs q}}$. Using Borell's inequality, see, e.g., \citep[Theorem 2.1]{Adler90}, 
\begin{equation}
\label{eq:2}
\log\prob{\exists{t\in[0,1]}: \bs W(t)>u\varepsilon\bs q}
\le\log\prob{\exists{t\in[0,1]}: Z(t)>u}
\le -\frac{(u-\mu)^2}{2\sigma^2},
\end{equation}
where $\mu=\ee\sup_{t\in[0,1]} Z(t)$ and $\sigma^2=\sup_{t\in[0,1]}\Var(Z(t))$. Combining \eqref{eq:1} and \eqref{eq:2} one can retrieve the extension hypothesis after proper optimization in
$\varepsilon\to 0$. Hence, the extended version of \autoref{thm:Mcase} implies that,
\begin{equation}
\label{eq:dummyq}
\lim_{u\toi}\frac{\sigma^2(u)}{u^2}\log\prob{\exists{t\in\rr_+}:\bs W(t)>u\bs q}= -\half\inf_{c\ge 0}\inf_{\bs x\ge c\bs q}\frac{\scal{S^{-1}(\bs x+\bs i(1))}{ C^{-1} S^{-1}(\bs x+\bs i(1))}}{c^{2-\gamma}}.
\end{equation}
The change of variable $c\mapsto t^{-1}$ gives
\begin{align*}
\inf_{c\ge 0}\inf_{\bs x\ge c\bs q}\frac{\scal{S^{-1}(\bs x+\bs i(1))}{ C^{-1} S^{-1}(\bs x+\bs i(1))}}{c^{2-\gamma}}
&=
\inf_{c\ge 0}\inf_{\bs x\ge \bs q}\frac{\scal{S^{-1}(\bs x+\bs i(c^{-1}))}{ C^{-1} S^{-1}(\bs x+\bs i(c^{-1}))}}{ c^{-\gamma}}\\
&=
\inf_{t\ge 0}\inf_{\bs x\ge \bs q}\frac{\scal{S^{-1}(\bs x+\bs i(t))}{ C^{-1} S^{-1}(\bs x+\bs i(t))}}{t^\gamma}.
\end{align*}
Therefore, \eqref{eq:dummyq} coincides with \citep[Proposition 2]{Debicki10b}. 

\newpage
\small\bibliography{biblioteczka.MLDP}
\end{document}